\documentclass[11pt]{article}

\usepackage{amsthm,amsmath,amssymb}
\usepackage{thmtools}					            
\usepackage{asypictureB}
\usepackage{asymptote}
\usepackage{ytableau}
\usepackage[vcentermath]{youngtab}
\usepackage{amsmath}
\usepackage{mathdots}
\usepackage{youngtab}
\usepackage[T1]{fontenc}
\usepackage{ifthen}
\usepackage{totpages}
\usepackage{graphicx}
\usepackage[letterpaper,margin=1in]{geometry}
\usepackage{pgffor}
\usepackage{xspace}
\usepackage{enumitem}
\usepackage{longtable}
\usepackage{booktabs}
\usepackage{colortbl}
\usepackage{comment}
\usepackage{mathtools}
\usepackage{multicol}
\usepackage{parskip}
\usepackage[hidelinks]{hyperref}
\usepackage{url}
\usepackage[english]{babel}
\usepackage{todonotes}

\usepackage{tikz}
\usetikzlibrary{arrows.meta}
\tikzset{>={Latex[width=3mm,length=3mm]}}

\newtheorem{theorem}{Theorem}
\newtheorem{lemma}[theorem]{Lemma}

\newtheorem{question}[theorem]{Question}
\newtheorem{corollary}[theorem]{Corollary}

\numberwithin{theorem}{section}
\numberwithin{lemma}{section}
\numberwithin{conjecture}{section}
\numberwithin{prop}{section}
\numberwithin{corollary}{section}
\numberwithin{question}{section}
\numberwithin{equation}{section}
\numberwithin{figure}{section}

\theoremstyle{definition}
\newtheorem{definition}[theorem]{Definition}
\newtheorem{example}[theorem]{Example}
\newtheorem{remark}[theorem]{Remark}

\newcommand{\SYT}{\mathrm{SYT}}

\title{Iterating the RSK Bijection}
\author{Maria Gillespie, Jacob Hocevar, Ananya Kulshrestha, Kosha Upadhyay}

\usepackage{youngtab}

\begin{document} 
\maketitle

\begin{abstract}

We investigate the dynamics of the well-known \textit{RSK bijection} on permutations when iterated on various reading words of the recording tableau.  In the setting of the ordinary (row) reading word, we show that there is exactly one fixed point per partition shape, and that it is always reached within two steps from any starting permutation.   

We also consider the modified dynamical systems formed by iterating RSK on the \textit{column reading word} and the \textit{reversed reading word} of the recording tableau.  We show that the column reading word gives similar dynamics to the row reading word.  On the other hand, for the reversed reading word, we always reach either a $2$-cycle or fixed point after two steps.  In fact, we reach a fixed point if and only if the shape of the initial tableau is self-conjugate.  
\end{abstract}

\section{Introduction}

A \textbf{permutation} of $1,2,\ldots,n$ is a rearrangement of the numbers $1,2,3,\ldots,n$ in some order.  A \textbf{standard Young tableau}, or SYT, is a way of writing the numbers $1,2,\ldots,n$ in the unit squares in the first quadrant such that the rows are left-aligned and increasing from left to right, and the columns are bottom-aligned and increasing from bottom to top.  

Robinson \cite{Robinson} famously discovered a bijection between permutations and pairs of standard Young tableaux of the same shape. (See Figure \ref{fig:example} for an example.)   Schensted and Knuth \cite{Schensted,Knuth} later generalized this bijection to a correspondence between words or two-line arrays and pairs involving \textit{semistandard} Young tableau, in which numbers may repeat.   Together, their discovery is often referred to as the Robinson-Schensted-Knuth bijection, or simply \textbf{RSK}. 

\begin{figure}
    \centering
    $5,1,4,6,2,3,7\hspace{0.5cm}\xrightarrow[]{RSK}\hspace{0.5cm}\left(\,\young(5,46,1237)\hspace{0.3cm},\hspace{0.3cm} \young(5,26,1347)\,\right)$
    \caption{A permutation and the pair of standard Young tableaux that it corresponds to under the RSK bijection.}
    \label{fig:example}
\end{figure}

The RSK bijection has great importance in many areas of mathematics, including representation theory, geometry, and combinatorics.  As one famous application, it gives a combinatorial proof of the identity $$\sum_{\lambda\vdash n} (f^\lambda)^2=n!$$ where $f^\lambda$ is the number of standard Young tableaux of shape $\lambda$. 

In the representation theory of the symmetric group $S_n$, this identity reflects the fact that the irreducible Specht module $V_\lambda$ has dimension $f^\lambda$, and the regular representation (of dimension $n!$) contains $f^\lambda$ copies of $V_\lambda$ for each shape $\lambda$ (see \cite{Sagan} for a thorough introduction to these topics).  Young tableaux and RSK-like insertion algorithms also have direct applications to the geometry of Grassmannians \cite[Chapter 10]{Fulton}, symmetric function theory \cite[Chapter 7]{Stanley}, the representation theory of Lie groups and algebras \cite[Chapter 8]{Fulton}, and crystal base theory \cite{BumpSchilling}.

The RSK algorithm on permutations takes a permutation $\pi$ and returns a pair $(S,T)$ of standard Young tableaux.  The tableau $S$ is called the \textbf{insertion tableau} and $T$ is called the \textbf{recording tableau}.  It is well-known that if one takes the \textbf{reading word} $w$ of $S$, formed by reading the rows in order from top to bottom, then applying the RSK algorithm again to $w$ will result in the same insertion tableau $S$ (and possibly a different recording tableau).  This leads naturally to the following question.

\begin{question}
 What happens when one applies the RSK algorithm again to the reading word of the recording tableau $T$?
\end{question}

In this paper, we investigate the above question by analyzing the dynamical system $f:S_n\to S_n$ on the set of permutations defined by applying RSK to a permutation and returning the reading word of the recording tableau. It is natural to ask what its fixed points, cycles, and orbits look like.

Our main results are as follows.  First, under iterating the dynamical system $f$, the shape of the RSK tableaux remains constant, and the orbits all hit fixed points within two steps.  We also prove that there exists exactly one fixed point per shape of SYT, and so the number of fixed points is the number $p(n)$ of partitions of $n$.  (See Section \ref{sect:fixed-points} for the full statements and proofs of these results.)

There are many natural variants of this iterative procedure that may be defined by alternative reading words.  One such variant is given by the \textbf{column reading word} formed by reading down the columns from left to right, which is known to have the same insertion tableau as the row reading word.  We show in Section \ref{sec:column} that iterating this process $c:S_n\to S_n$, given by taking RSK of a permutation and then taking the column reading word of the recording tableau, also results in a fixed point in at most two steps, the shape is invariant, and there is one fixed point per shape.

In Section \ref{sec:reverse}, we also consider the process of iterating RSK on the \textit{reverse} reading word of the recording tableau, which has more interesting dynamics.  We show that in this case, the orbits enter a $1$- or $2$-cycle in at most two steps, and we classify these $2$-cycles and fixed points completely.

\section{Background and notation}\label{sec:background}
 We begin by recalling some basic combinatorial definitions and notations.
 \subsection{Permutations}
 
 A \textbf{permutation} of size $n$ is an ordering of the set $\{1, 2, 3, \ldots, n \}$.  We denote a permutation with commas separating the numbers, as in $$2,  4,  7,  3,  5,  1,  6,  8.$$
 For example, the permutations of size $3$ are 
 \begin{center}
 \begin{tabular}{ccc}
 1,2,3 & 2,1,3 & 3,1,2 \\
 1,3,2 & 2,3,1 & 3,2,1
 \end{tabular}
 \end{center} 
for a total of $3! = 6$ permutations.  In general, there are $n!$ 
permutations of size $n$.

\begin{definition}\label{def:inverse}
The \textbf{inverse} of a permutation $\pi$ is the permutation $\sigma$ of $1,2,\ldots,n$ for which $\sigma_i=j$ if and only if $\pi_j=i$.  We write $\pi^{-1}$ to denote the inverse of $\pi$.
\end{definition}

For example, if $\pi$ is the permutation $2,3,1$, then the first element of the inverse would be $3$, because $1$ is in the \textit{third} position. The next element is $1$, because $2$ is in the \textit{first} position. The final element is $2$, because $3$ is found in the \textit{second} position. Therefore, the inverse permutation of $2,3,1$ is $3,1,2$.

\begin{remark}
 If we think of the permutation $\pi_1,\ldots,\pi_n$ as the bijection from $\{1,2,\ldots,n\}$ to itself that sends $i$ to $\pi_i$ for all $i$, then the inverse permutation corresponds to the inverse function.
\end{remark}

\begin{definition}\label{def:involution}
An \textbf{involution} is a permutation whose inverse is itself.
\end{definition}

\begin{example}
The inverse of the permutation $2,1,3$ is $2,1,3$ (itself), so $2,1,3$ is an involution.
\end{example}

\subsection{Partitions and Young tableaux}

A \textbf{partition} of a positive integer $n$ is a sequence $\lambda=(\lambda_1,\lambda_2,\ldots,\lambda_k)$ of positive integers satisfying $\lambda_1\ge \lambda_2\ge \cdots \ge \lambda_k$ and $\sum_i \lambda_i=n$.  If $\lambda$ is a partition of $n$ we write $\lambda\vdash n$.

We draw a partition $\lambda$ using a \textbf{Young diagram} in French notation by stacking $\lambda_i$ left-justified boxes on the $i$-th row from the bottom of the diagram for each $i$.

\begin{definition}
A \textbf{standard Young tableau}, or \textbf{SYT}, is a filling of the boxes of a Young diagram of size $n$ with the numbers $1,2,\ldots,n$, such that the numbers in each row are increasing from left to right, and the numbers in each column are increasing from bottom to top.
\end{definition}

\begin{example}\label{ex:tableau}
The diagram below is an example of a standard Young tableau of size $8$, filling the partition shape $\lambda=(4,2,2)$.
$$\young(46,25,1378)$$
\end{example}

\begin{definition}
The \textbf{reading word} of an SYT is the permutation formed by concatenating the rows from top to bottom. In other words, the reading word is obtained by reading the numbers from left to right, then top to bottom, as you would words on a page.
\end{definition}

For example, the reading word of the tableau in Example \ref{ex:tableau} is $4,6,2,5,1,3,7,8$.

\begin{remark}
Not every permutation is a reading word. For example, any permutation that starts with $1$ (excluding $1,2,3,\ldots,n$) cannot be the reading word of any tableau because the $1$ is anchored to the bottom row.
\end{remark}

\begin{definition} The \textbf{transpose} of a Young diagram or tableau is the result of reflecting the shape or tableau about the diagonal starting from the bottom-left corner. For example, the transpose of the tableau at left is the tableau at right:
$$\young(4,2,13) \hspace{2cm} \young(3,124)$$
\end{definition}

\subsection{The RSK bijection on permutations}

The RSK bijection assigns to each permutation $\pi$ of size $n$ a unique pair $\mathrm{RSK}(\pi)=(S,T)$ of standard Young tableaux, known as the \textbf{insertion tableau} and \textbf{recording tableau}. 


\subsubsection{Insertion tableau}

To define the insertion tableau corresponding to a permutation $\pi=\pi_1,\pi_2,\ldots,\pi_n$, we read the permutation from left to right and insert the $i$-th term $\pi_i$ into the tableau  on the $i$-th step (starting with the empty tableau on step $0$). 

\textbf{Case 1:} When we are inserting $\pi_i$ into the SYT, if it is bigger than any box in the bottom row, then we place it at the end of that row. For example, the permutation $1, 2, 3 $ gives us the sequence of tableaux $$\emptyset \to \raisebox{-0.08cm}{\young(1)}\to \raisebox{-0.08cm}{\young(12)}\to \raisebox{-0.08cm}{\young(123)}$$
since each number is bigger than all the previous numbers.  The final tableau, $\young(123)$, is the insertion tableau of the permutation $1,2,3$.

\textbf{Case 2:} If instead $a=\pi_i$ is not larger than every entry in the first row, then a different process applies.  We define $b$ to be the leftmost box in the first row that is greater than $a$.  This box $b$ will get replaced by $a$.  Now, we insert box $b$ into the second row following the same rules as above; putting it at the end of the second row if larger, or replacing some box $c$ again. This process continues as necessary until all boxes are placed in some row.

We then repeat the algorithms given by the two cases until all the numbers $\pi_1,\ldots,\pi_n$ have been inserted, and the result is the insertion tableau.

For example, given the permutation $2,3,1 $, we first place the $2$ and $3$ giving
$$\young(23)$$ 

Then, we must replace the $2$ with a $1$ since $2$ is the first number greater than $1$ from left-to-right on the bottom row.
The $2$ is then placed at the end of the second row to obtain the insertion tableau $$\young(2,13)$$

\subsubsection{Recording tableau}

The recording tableau of the RSK bijection is not generated from the permutation directly, but rather from the steps taken when producing the insertion tableau. For its construction, we look at the shape of the insertion tableau and record the number of the step in which each box was added.  More precisely:

 \begin{definition}
 The \textbf{recording tableau} for a permutation $\pi$ is defined as follows.  When inserting $\pi_i$ using the insertion algorithm above, the shape increases by exactly one box; this box is labeled by $i$ in the recording tableau.
 \end{definition}
   
 Note that the definition above guarantees that the insertion and recording tableaux have the same shape. For example, consider the permutation $2, 4, 7, 3, 5, 1, 6, 8$. The RSK algorithm applied to this permutation is illustrated in the table in Figure \ref{bigtable}.

\begin{figure}[h]
\begin{center}
\renewcommand{\arraystretch}{2}
\begin{tabular}{c|l|l}
    Step Number & Insertion Tableau & Recording Tableau \\ \hline
    1 & \young(2) & \young(1) \\
    2 & \young(24) & \young(12) \\
    3 & \young(247) & \young(123) \\
    4 & \young(4,237) & \young(4,123) \\
    5 & \young(47,235) & \young(45,123) \\
    6 & \young(4,27,135) & \young(6,45,123) \\
    7 & \young(4,27,1356) & \young(6,45,1237) \\
    8 & \young(4,27,13568) & \young(6,45,12378) 
\end{tabular}
\end{center}

\caption{\label{bigtable} The steps of the RSK algorithm on the permutation $2,4,7,3,5,1,6,8$.}
\end{figure}

Putting these definitions together, let $S_n$ be the set of all permutations of $1,2,\ldots,n$, and let $\SYT(\lambda)$ be the set of all standard Young tableaux of shape $\lambda$.  Then the RSK bijection is the map $$\mathrm{RSK}:S_n\to \bigcup_{\lambda\vdash n} \SYT(\lambda)\times \SYT(\lambda)$$ such that $\mathrm{RSK}(\pi)=(S,T)$ where $S$ is the insertion tableau and $T$ is the recording tableau of the permutation $\pi$. Thus, for example, we have that $\mathrm{RSK}(2,4,7,3,5,1,6,8)$ is the pair of tableaux
$$\left(\,\young(4,27,13568)\hspace{0.3cm},\hspace{0.3cm}\young(6,45,12378)\,\right).$$

It was shown by Robinson \cite{Robinson} that the map $\mathrm{RSK}$ above is a bijection.


\subsection{RSK and reading words}

In the example above, the insertion tableau (left) has reading word  $4,2,7,1,3,5,6,8$. The recording tableau (right) has a reading word of $6,4,5,1,2,3,7,8$.  

As mentioned in the introduction, applying RSK to the reading word of the insertion tableau yields the same insertion tableau.  Indeed, applying RSK to the permutation $4,2,7,1,3,5,6,8$ yields the pair of tableaux
$$\left(\,\young(4,27,13568) \hspace{0.3cm},\hspace{0.3cm}\young(4,25,13678)\,\right),$$
which has the same insertion tableau as above (but a different recording tableau).  We state this fact precisely as follows.

\begin{theorem}[{\cite[Theorems A.1.1.4, A.1.1.6]{Stanley}}]\label{insertion-same} 
Let $S$ be a standard Young tableau with reading word $s$.  Then applying the RSK algorithm to $s$ will generate the same insertion tableau $S$. 
\end{theorem}

In particular, repeating the process of applying RSK to the reading word of the insertion tableau will therefore reach a fixed point in just one step.  However, the same is not true of the recording tableau, as we will see in Section \ref{sect:fixed-points}.

Finally, we recall how RSK interacts with inverse permutations.  

\begin{theorem}[{\cite[Theorem 7.13.1]{Stanley}}]\label{inverse}
If $RSK(\pi) = (S, T)$, where $S$ and $T$ are the insertion and recording tableaux respectively, then $RSK(\pi^{-1}) = (T, S)$.
\end{theorem} 

The following corollary is a known fact, but we state it here and provide a brief proof for the reader's convenience.

\begin{corollary}\label{lem:involution}
A permutation is an involution if and only if its insertion and recording tableaux are identical.
\end{corollary}

\begin{proof}
  Let $\pi$ be an involution. Applying the RSK algorithm to $\pi$ will produce an SYT pair $(S,T)$. Then $\mathrm{RSK}(\pi^{-1})=(T,S)$ by Theorem \ref{inverse}.
  Because $\pi$ is an involution, by Definition \ref{def:involution}, $\pi^{-1} = \pi$. Therefore, $\mathrm{RSK}(\pi^{-1})=\mathrm{RSK}(\pi)$, so $(T,S)=(S,T)$ as pairs. Therefore, $S=T$.
\end{proof}

\subsection{The shape of RSK}

It turns out that the partition shape $\lambda$ of the insertion tableau of a permutation is entirely determined by certain statistics involving subsequences.

\begin{definition}
A \textbf{subsequence} of a sequence $w$ is a sequence obtained by deleting some or no elements of $w$ without changing the order of the remaining elements.
\end{definition}

\begin{definition}\label{def-increase-decrease}
Let $w$ be a permutation. For a positive integer $k$, we say that a subsequence of $w$ is \textbf{$k$-decreasing} if it can be decomposed into $k$ disjoint decreasing subsequences -- equivalently, if it does not contain an increasing subsequence of length $k + 1$. For example, $$(7, 8, 6, 2, 3, 1, 5, 4)$$ is 3-decreasing since it can be decomposed into subsequences $(7, 2, 1)$, $(8, 6, 3)$, and $(5,4)$. We denote the maximum length of a $k$-decreasing subsequence of $w$ by $D_k(w)$. In this example, $D_2(w) = 7$ since then we can take the subsequence $(7, 8, 6, 2, 1, 5, 4)$, which can be decomposed into $2$ subsequences $$(7, 6, 5, 4) \hspace{1cm} (8, 2, 1)$$ 
Similarly, we say that a subsequence of $w$ is \textbf{$k$-increasing} if it can be decomposed into $k$ disjoint increasing subsequences -- equivalently, if it does not contain an decreasing subsequence of length $k + 1$. We denote the maximum length of a $k$-decreasing subsequence of $w$ by $I_k(w)$.
\end{definition}

\begin{theorem}[{\cite[Theorem 7.23.13]{Stanley}}]\label{thm:shape}
Suppose the permutation $w$ maps under RSK to a pair of tableaux having partition shape $\lambda$.  Then for any $j$, we have that $$\lambda_1+\cdots+\lambda_j=I_j(w).$$  Similarly, if $\lambda'$ is the transpose partition, then $$\lambda_1'+\cdots+\lambda_j'=D_j(w).$$
\end{theorem}
(Note that $\lambda_i'$ can be interpreted as the length of the $i$-th column of the diagram of $\lambda$.)

 For instance, consider the permutation $2, 4, 7, 3, 5, 1, 6, 8$ from our running example, which maps to the pair: $$\left(\,\young(4,27,13568) \hspace{0.3cm},\hspace{0.3cm} \young(6,45,12378)\,\right)$$ The longest increasing subsequence of the permutation has length $5$ (e.g., $(2, 3, 5, 6, 8)$, though this is not a unique choice). The longest decreasing subsequence has length $3$ (e.g., $(7, 5, 1)$) and these facts are reflected in the observation that the tableaux have width $5$ and height $3$. The longest $2$-decreasing subsequence has length $5$ (e.g., $(4,7,3,5,1)$, which can be decomposed into $(4,3)$ and $(7,5,1)$) which corresponds to the fact that the first two columns contain a total of $5$ boxes.

\section{Iterating on the reading word of the recording tableau}\label{sect:fixed-points}

We now study the dynamical system obtained by iterating the RSK algorithm on the reading word of the recording tableau.

\begin{definition}
 Let $S_n$ be the set of all permutations of $1,2,\ldots,n$ and define $f:S_n\to S_n$ as follows.  For a permutation $\pi\in S_n$, let $\mathrm{RSK}(\pi)=(S,T)$ and let $w$ be the reading word of $T$. Then we set $f(\pi)=w$.
\end{definition}

\tikzstyle{permblock} = [draw, text centered, minimum height=1em]
\tikzstyle{block}  = [rectangle, draw, text width=3.5cm, text centered, minimum height=1em]
\tikzstyle{lblock} = [rectangle, draw, text width=5cm, text centered, minimum height=1em]
\tikzstyle{rblock} = [rectangle, draw, text width=3.5cm, text centered, minimum height=1em]

\begin{figure}
    \centering
    
\begin{tikzpicture}[node distance=2cm]
\node (shape131) {$\yng(1,1,2)$};
\node (rw11) [permblock, right of=shape131, xshift=0.7cm] {$3,2,4,1$};
\node (rw12) [permblock, right of=rw11, xshift=0.3cm] {$4,2,3,1$};
\node (rw13) [permblock, right of=rw12, xshift=0.3cm] {$4,1,3,2$};
\node (rw14) [permblock, right of=rw13, xshift=0.7cm] {$2,4,3,1$};
\node (rw15) [permblock, right of=rw14, xshift=0.3cm] {$1,4,3,2$};
\node (rw16) [permblock, right of=rw15, xshift=0.3cm] {$3,4,2,1$};
\node (shape132)[below of=shape131]{$\yng(1,1,2)$};
\node (rw21) [permblock, right of=shape132, xshift=3cm] {$4,2,1,3$};
\node (rw22) [permblock, right of=rw21, xshift=5.3cm] {$4,3,1,2$};
\node (shape133)[below of=shape132]{$\yng(1,1,2)$};
\node (rw31) [permblock, right of=shape133, xshift=6.65cm] {$3,2,1,4$};

\draw[->] (rw11) -- (rw21) ;
\draw[->] (rw12) -- (rw21) ;
\draw[->] (rw13) -- (rw21) ;
\draw[->] (rw14) -- (rw22) ;
\draw[->] (rw15) -- (rw22) ;
\draw[->] (rw16) -- (rw22) ;
\draw[->] (rw21) -- (rw31) ;
\draw[->] (rw22) -- (rw31) ;
\draw[->] (rw31.south)arc(-158:158:7mm);

\node (shape134)[below of=shape133]{$\yng(1,3)$};
\node (rw41) [permblock, right of=shape134, xshift=0.7cm] {$2,3,1,4$};
\node (rw42) [permblock, right of=rw41, xshift=0.3cm] {$1,3,2,4$};
\node (rw43) [permblock, right of=rw42, xshift=0.3cm] {$1,4,2,3$};
\node (rw44) [permblock, right of=rw43, xshift=0.7cm] {$1,3,4,2$};
\node (rw45) [permblock, right of=rw44, xshift=0.3cm] {$1,2,4,3$};
\node (rw46) [permblock, right of=rw45, xshift=0.3cm] {$2,3,4,1$};
\node (shape135)[below of=shape134]{$\yng(1,3)$};
\node (rw51) [permblock, right of=shape135, xshift=3cm] {$3,1,2,4$};
\node (rw52) [permblock, right of=rw51, xshift=5.3cm] {$4,1,2,3$};
\node (shape136)[below of=shape135]{$\yng(1,3)$};
\node (rw61) [permblock, right of=shape136, xshift=6.65cm] {$2,1,3,4$};

\draw[->] (rw41) -- (rw51) ;
\draw[->] (rw42) -- (rw51) ;
\draw[->] (rw43) -- (rw51) ;
\draw[->] (rw44) -- (rw52) ;
\draw[->] (rw45) -- (rw52) ;
\draw[->] (rw46) -- (rw52) ;
\draw[->] (rw51) -- (rw61) ;
\draw[->] (rw52) -- (rw61) ;
\draw [->] (rw61.south)arc(-158:158:7mm);

\node (shape221)[below of=shape136]{$\yng(2,2)$};
\node (rw71) [permblock, right of=shape221, xshift=3cm] {$2,1,4,3$};
\node (rw72) [permblock, right of=rw71, xshift=5.3cm] {$3,1,4,2$};
\node (shape222)[below of=shape221]{$\yng(2,2)$};
\node (rw81) [permblock, right of=shape222, xshift=6.65cm] {$2,4,1,3$};
\node (shape223)[below of=shape222]{$\yng(2,2)$};
\node (rw91) [permblock, right of=shape223, xshift=6.65cm] {$3,4,1,2$};
\draw[->] (rw71) -- (rw81) ;
\draw[->] (rw72) -- (rw81) ;
\draw[->] (rw81) -- (rw91) ;
\draw [->] (rw91.south)arc(-158:158:7mm);

\node (shape041)[below of=shape223]{$\yng(4)$};
\node (rw101) [permblock, right of=shape041, xshift=6.65cm] {$1,2,3,4$};
\node (shape401)[below of=shape041]{$\yng(1,1,1,1)$};
\node (rw111) [permblock, right of=shape401, xshift=6.65cm] {$4,3,2,1$};]
\draw [->] (rw101.south)arc(-158:158:7mm);
\draw [->] (rw111.south)arc(-158:158:7mm);
\end{tikzpicture}
\caption{The directed graph of the discrete dynamical system formed by iterating $f$ on all permutations of size $4$.  The shape of the insertion tableau of each permutation is shown at left.}
    \label{fig:first-graph}
\end{figure}
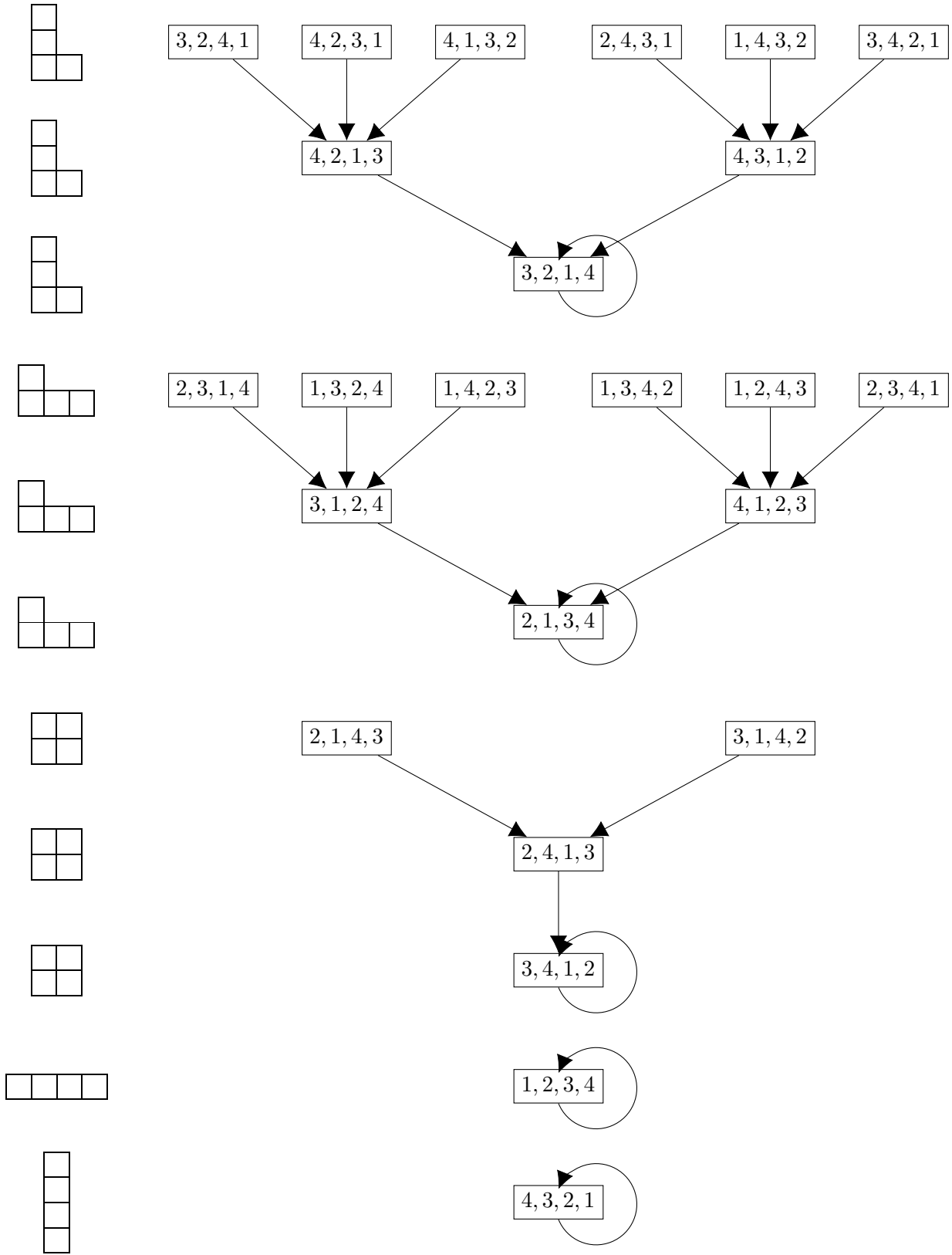

The map $f$ on $S_4$ is illustrated in Figure \ref{fig:first-graph}.  Notice that every permutation reaches a \textbf{fixed point} (a permutation $\pi$ for which $f(\pi)=\pi$) in at most two iterations of $f$.  Proving this fact is the main goal of this section.

We first show that the shape of the tableaux assigned to a permutation is invariant under $f$.  

\begin{lemma}\label{lem:constant-shape}
Starting with a pair of SYT $(S,T)$ of the same shape, let $(S',T')$ be the result of applying the RSK algorithm to the reading word of $T$.  Then $(S',T')$ have the same shape as $(S,T)$.
\end{lemma}

\begin{proof}
Because a recording tableau is defined to copy the shape of its insertion tableau at each step, we know that $S$ and $T$ both have some shape $\lambda$, and that $S'$ and $T'$ have some shape $\mu$. By Theorem \ref{insertion-same}, we know that using the reading word of $T$ to construct a new insertion tableau results in the same tableau $T$. Therefore $T=S'$, and thus $\lambda = \mu$.
\end{proof}

This result is useful because it allows us to limit our investigation to tableaux of a fixed shape $\lambda$.

We now show that, given a permutation that is the reading word of some tableau, the recording tableau is uniquely determined.

\begin{lemma}\label{lem:reading-recording}
Let $\lambda=(\lambda_1,\ldots,\lambda_k)$ be a partition.  Then there is a fixed tableau $T_\lambda$ such that, when RSK is applied to the reading word of any SYT of shape $\lambda$, the recording tableau is always $T_\lambda$.

In particular, $T_\lambda$ is the tableau formed by filling the bottom entries of the first $\lambda_k$ columns with $1,2,\ldots,\lambda_k$, then filling the bottommost empty squares of the first $\lambda_{k-1}$ columns with $\lambda_k+1,\ldots,\lambda_k+\lambda_{k-1}$, and so on.  
\end{lemma}

\begin{remark}\label{rem:T-lambda}
Another way of thinking of $T_\lambda$ is as the tableau formed by first applying ``upwards gravity'' to the boxes in $\lambda$ to make them sit on the ceiling, then filling the rows with $1,2,\ldots,n$ in order from bottom to top (and left to right across each row), and finally letting gravity pull the boxes down to the floor again.  Below is an example of $T_{(4,3,2)}$ at left and, at right, the result of applying upwards gravity to the boxes:
$$\young(67,348,1259) \hspace{2cm} \young(6789,345,12)$$
\end{remark}

We now prove Lemma \ref{lem:reading-recording}.

\begin{proof}
Let $w$ be the reading word of an SYT $S$ of shape $\lambda$.  Since we read the rows of $S$ from top to bottom, the top row (of length $\lambda_k$) will be inserted first to form a row of $\lambda_k$ entries, making the first $\lambda_k$ squares on the bottom row of the recording tableau of $\mathrm{RSK}(w)$ be labeled $1,2,\ldots,\lambda_k$.  We then insert the second row of $S$, which bumps up all elements that we already inserted (since the columns of $S$ are increasing) and then puts the rest of the second row to the right, which gives the next round of entries in the recording tableau as described.  This process continues in the same fashion for each row inserted, yielding recording tableau $T_\lambda$.
\end{proof}

\begin{theorem}\label{thm:eventually-fixed}
  For any permutation $\pi$, two or fewer applications of the map $f$ will result in a fixed point.
\end{theorem}

\begin{proof}
  Let $\pi$ be an arbitrary permutation. Apply the RSK algorithm to $\pi$ to obtain the SYT pair $(S,T)$. Take the reading word of $T$ and produce a new SYT pair.  By Theorem \ref{insertion-same} and Lemma \ref{lem:constant-shape}, this pair must be $(T,R)$, where $R$ is the same shape as $T$, say shape $\lambda$. By Lemma \ref{lem:reading-recording}, we know that $R$ is the unique recording tableau $T_\lambda$ generated by all of the permutations which are the reading words of tableaux of shape $\lambda$.  
  
  Applying the RSK algorithm to the reading word of $R=T_\lambda$ will then yield a pair with $R$ as the insertion tableau (again by Theorem \ref{insertion-same} and Lemma \ref{lem:constant-shape}). But this then generates $R=T_\lambda$ as the recording tableau again by Lemma \ref{lem:reading-recording}. Thus its pair of SYT are $(R,R)$. Additionally, because it was generated from the reading word of $R$, it is a fixed point.
\end{proof}

Because all permutations lead into a fixed point within two applications of $f$, we can now classify all cycles of the dynamical system. We first recall the definition of a cycle.

\begin{definition}\label{def:cycle}
A \textbf{cycle} of a dynamical system $g:A\to A$ is a sequence of elements $$a,g(a),g(g(a)),\ldots, g^n(a)$$ such that the output $g^n(a)$ is equal to the initial input $a$. If it takes a minimum of $n$ steps to return to the input, we say the cycle is an  \textbf{$n$-cycle}, and if $n>1$ we say the cycle is \textbf{nontrivial}.
\end{definition}

We now have the following corollary.

\begin{corollary}
  The map $f$ has no nontrivial cycles.
\end{corollary}
  
These two results completely describe the dynamics of the iterated process $f$.  To fully understand the combinatorics of the system, we now classify and analyze the fixed points themselves. 

\begin{theorem}\label{thm:one-fixed}
The map $f$ has exactly one fixed point of each possible Young diagram shape.
\end{theorem}
\begin{proof}
The proof of Theorem \ref{thm:eventually-fixed} shows that every initial permutation $\pi$ whose RSK tableaux have shape $\lambda$ eventually maps to the pair $(T_\lambda,T_\lambda)$ where $T_\lambda$ is defined as in Lemma \ref{lem:reading-recording}.  The permutation $\pi_\lambda$ associated to this pair $(T_\lambda,T_\lambda)$ must therefore be the unique fixed point, since by Theorem \ref{thm:eventually-fixed}, if there were a different fixed point for shape $\lambda$, it would eventually map to this one, and therefore equals $\pi_\lambda$.
\end{proof}

This means we have a precise enumeration of the fixed points of the system as well.

\begin{corollary}
  The number of fixed points of $f$ on $S_n$  is equal to the number of partitions of $n$.
\end{corollary}

We now analyze the properties of the fixed point permutation $\pi_\lambda$.  The following is an immediate result of the proof of Theorem \ref{thm:eventually-fixed}, so we omit the proof.

\begin{lemma}\label{lem:fixed-point}
The unique fixed point $\pi_\lambda$ for shape $\lambda$ satisfies $\mathrm{RSK}(\pi_{\lambda})=(T_\lambda,T_\lambda)$, and the reading word of $T_\lambda$ is $\pi_\lambda$.
\end{lemma}

By Lemmas \ref{lem:fixed-point} and \ref{lem:involution}, we can conclude the following.
\begin{corollary}
  Each fixed point $\pi_\lambda$ of $f$ is an involution.
\end{corollary}

Finally, we describe geometrically which involutions these permutations are.  Since $\pi_\lambda$ is the reading word of $T_\lambda$, we can write it down explicitly for a given shape.  For instance, in the example of $T_{(4,3,2)}$ shown in Remark \ref{rem:T-lambda}, the reading word is $673481259$.  We see that this is the involution that switches $1$ and $6$, switches $7$ and $2$, and switches $8$ and $5$.  Comparing this to the tableau $T_{(4,3,2)}$:
$$\young(67,348,1259)$$ we see that it is the involution that pairs off the entries by matching each column with itself upside-down (so in this case $3,4$, and $9$ match with themselves and are therefore unchanged by the involution).

In general this property will hold as follows.  The top row of $T_\lambda$ contains the entries $$(\lambda_k+\lambda_{k-1}+\cdots+\lambda_{2})+1,\ldots,(\lambda_k+\cdots+\lambda_{2})+\lambda_k,$$ and these are the first $\lambda_k$ entries of the reading word $\pi_\lambda$.  Meanwhile the numbers $1,2,\ldots,\lambda_k$ are in \textit{positions}
$$(\lambda_k+\lambda_{k-1}+\cdots+\lambda_{2})+1,\ldots,(\lambda_k+\cdots+\lambda_{2})+\lambda_k$$ in the reading word $\pi_\lambda$. Thus these numbers are indeed swapped with $1,2,\ldots,\lambda_k$ in the involution, and they are the top and bottom entries of the first $\lambda_k$ columns.  We similarly can then pair off the second-to-top row of entries with the numbers $\lambda_k+1,\ldots,\lambda_k+\lambda_{k-1}$, and so on.

\section{Variations on the reading word} \label{sect:reading-word}

We now consider what happens if we modify the dynamical system by using variants on the reading word.  We discuss two variants: the \textit{column reading word} and the \textit{reversed reading word}.

\subsection{Column reading word}\label{sec:column}

The \textbf{column reading word} of a standard Young tableau $T$ is the permutation formed by reading down the columns in order from left to right.  For instance, the column reading word of $$\young(4,257,136)$$ is $4,2,1,5,3,7,6$.
It is known (see \cite[Section 2.3]{Fulton}) that inserting the column reading word yields the original tableau, just as in the case of the (row) reading word.  In this case it is even easier to see what the recording tableau will be when we apply RSK to the column reading word of a tableau.

\begin{example}
If we apply RSK to the column reading word $4,2,1,5,3,7,6$ above, first the $4,2,1$ are inserted and form a column, yielding the pair: $$\young(4,2,1)\hspace{2cm} \young(3,2,1)$$ Then, inserting the $5$ and $3$ yields the next column:
$$\young(4,25,13)\hspace{1.5cm}\young(3,25,14)$$
and finally inserting the $7$ and $6$ forms the final column:
$$\left(\,\young(4,257,136)\hspace{0.3cm},\hspace{0.3cm} \young(3,257,146)\,\right)$$
\end{example}

From the example above, and using proof techniques similar to that of Lemma \ref{lem:reading-recording}, we obtain the following lemma, whose proof we omit.

\begin{lemma}\label{lem:col-reading}
For any partition $\lambda$, there is a fixed tableau $Q_\lambda$ such that, when RSK is applied to the column reading word of any SYT of shape $\lambda$, the recording tableau is always $Q_\lambda$.

In particular, $Q_\lambda$ is the tableau formed by filling the first column with $1,2,\ldots,\lambda_1'$ (where $\lambda_1'$ is the length of the first column, or first row of the transpose shape $\lambda'$), then filling the second column with  $\lambda_1'+1,\ldots,\lambda_1'+\lambda_2'$, and so on.  
\end{lemma}

By an identical analysis to that in Section \ref{sect:fixed-points}, we can then obtain the following results, which we state as one summary theorem.

\begin{theorem}
 Let $c:S_n\to S_n$ be the operation of applying RSK to a permutation $\pi$ to get a pair $(S,T)$ and then taking the \textit{column} reading word of the recording tableau $T$.  Then:
 \begin{itemize}
     \item The map $c$ preserves the shape of the Young tableaux obtained under RSK at each step.
     \item If $c$ is iterated starting with any permutation $\pi\in S_n$, it reaches a fixed point in at most two steps.
     \item There is exactly one fixed point $\sigma_\lambda$ for each partition $\lambda$, and its RSK insertion pair is $(Q_\lambda,Q_\lambda)$.  Moreover, $\sigma_\lambda$ is the column reading word of $Q_\lambda$.
 \end{itemize}
\end{theorem}

This theorem shows that, in some sense, the iterated process using the column reading word is more straightforward than that of the row reading word.  Indeed, $Q_\lambda$ is easier to define than $T_\lambda$.  Interestingly, the involution $\sigma_\lambda$ is still obtained by taking the columns of $Q_\lambda$ and switching the opposite entries (reflected about the horizontal) in pairs in each column, in the same manner that $\pi_\lambda$ is obtained from $T_\lambda$.

\subsection{Reversed reading word}\label{sec:reverse}

We now consider the more subtle case of the \textit{reverse} of the reading word, formed by reading the word from right to left.  As an example, the reverse of $3,1,5,2,4$ is $4,2,5,1,3$.

\begin{definition}
Define $r:S_n\to S_n$ as follows.  For a permutation $\pi \in S_n$, let $\mathrm{RSK}(\pi)=(S,T)$ and let $w'$ be the reversed reading word of $T$.  Then we set $r(\pi)=w'$.
\end{definition}

This operation is fundamentally different from using the row and column reading words, because it does not generate the same insertion tableau.
In fact, the shape is no longer even preserved, as seen by the following lemma.

\begin{lemma} \label{lem:transposition}
Suppose a standard Young tableau $T$ has reading word $w$. Let $w'$ be the reverse of the string $w$. Then the insertion tableau of the application of RSK to $w'$ has a shape which is the \text{transpose} of the shape of $T$.
\end{lemma}

\begin{proof}
Let $w$ be the reading word $a_1,a_2,a_3,\ldots,a_n$. Thus, $w' = a_n, a_{n-1}, a_{n-2}, \ldots a_1$.  Let $a_{i_1} > a_{i_2} >  \cdots >a_{i_k}$ be a decreasing subsequence of $w$.  Then in $w'$, we have $a_{i_k} < a_{i_{k-1}} < \cdots < a_{i_2} < a_{i_1}$ is an increasing subsequence.  Similarly, increasing subsequences of $w$ become decreasing subsequences of $w'$.

Let the length of the longest $k$-increasing sequence in $w$ be $x$ and the length of the longest $k$-decreasing sequence in $w$ be $y$. Then the longest $k$-decreasing subsequence for $w'$ would have length $x$ and the longest $k$-increasing subsequence for $w'$ would be $y$. Thus, by Theorem \ref{thm:shape}, the rows and columns of $\mathrm{RSK}(w)$ and $\mathrm{RSK}(w')$ would be swapped, resulting in a transposed shape.
\end{proof}

We now show that in fact something much stronger holds: the insertion tableau of the reversed reading word of any tableau is simply the transposed tableau itself.

\begin{theorem}\label{theorem:transpose-insertion}
  Let $T$ be a standard Young tableau with reading word $w$.  Then inserting the reversed word $w'$ yields the transposed tableau $T'$.
\end{theorem}

\begin{proof}
Suppose $T$ has reading word $w = w_1,w_2,w_3, \ldots ,w_n$, as illustrated below. The shape of the tableau is a partition of $n$.  Let $\lambda$ be the shape of the tableau $T$, with $\lambda = (\lambda_1,\ldots,\lambda_k)$ where $\lambda_i$ is the size of the $i$-th row from the bottom. We define $$\sigma_{j} = \sum_{i=0}^{j-1} \lambda_{k-i}$$ so that $\sigma_0 = 0$, $\sigma_1 = \lambda_k$, $\sigma_2=\lambda_k+\lambda_{k-1}$, and so on up to $\sigma_k = \lambda_k+\lambda_{k-1}+\cdots+\lambda_1=n$.

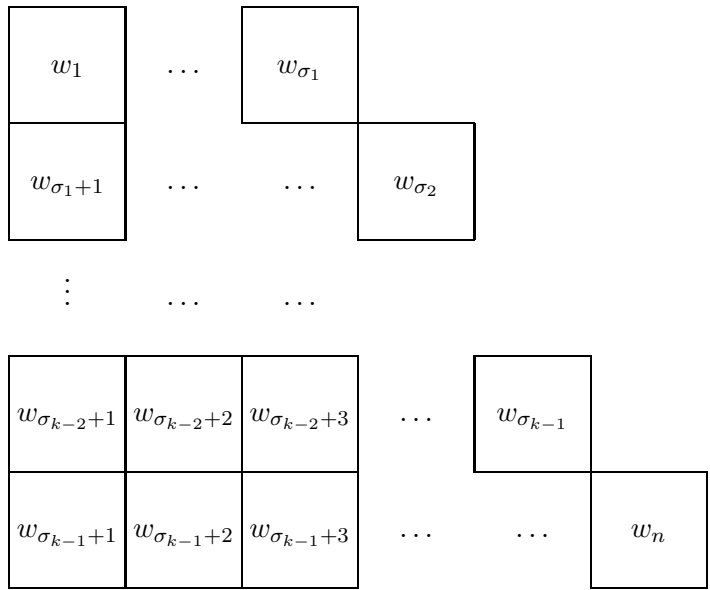
\[
\ytableausetup{mathmode, boxsize=4em}
\begin{ytableau}
w_{1} & \none[\dots] & w_{\sigma_1}\\
w_{\sigma_1 + 1} & \none[\dots] & \none[\dots] & w_{\sigma_2}\\
\none[\vdots] & \none[\dots] & \none[\dots] \\
w_{\sigma_{k-2} + 1} & w_{\sigma_{k-2} + 2}  & w_{\sigma_{k-2} + 3} & \none[\dots] & w_{\sigma_{k-1}}\\
w_{\sigma_{k-1} + 1} & w_{\sigma_{k-1} + 2} & w_{\sigma_{k-1} + 3} & \none[\dots] & \none[\dots] & w_{n} \\
\end{ytableau}
\]
    We would first insert $w_n$, which yields the tableau:
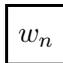
\[
\ytableausetup{mathmode, boxsize=2em}
\begin{ytableau}
w_{n}\\
\end{ytableau}
\]

\noindent As we proceed to insert $w_{n-1}, \ldots w_{\sigma_k+1}$, we notice, by the definition of a standard Young tableau, that $w_n > w_{n-1} > \cdots > w_{\sigma_k+1}$. Hence, by the RSK algorithm, when a new box from this row is inserted, all the previously present boxes bump out and the resulting tableau from inserting the first $\lambda_1$ numbers of $w'$ is a single column:

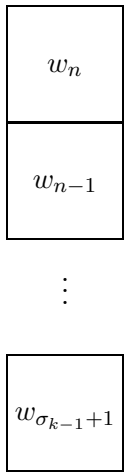
\[
\ytableausetup{mathmode, boxsize=4em}
\begin{ytableau}
w_{n}\\
w_{n-1}\\
\none[\vdots]\\
w_{\sigma_{k-1} + 1}
\end{ytableau}
\]

 We now continue this process for the next $\lambda_{2}$ entries of $w'$.
When we insert $w_{\sigma_{k-1}}$ we will not bump anything out as $$w_{\sigma_{k-1}} > \cdots > w_{\sigma_{k-2} + 1} > w_{\sigma_{k-1} + 1}$$
by semistandardness.  Thus we have, at the next step:
\[
\ytableausetup{mathmode, boxsize=4em}
\begin{ytableau}
w_{n}\\
w_{n-1}\\
\none[\vdots]\\
w_{\sigma_{k-1} + 1} & w_{\sigma_{k-1}}
\end{ytableau}
\]

 Next we insert $w_{\sigma_{k-1} - 1}$ which would bump out $w_{\sigma_{k-1}}$ since $w_{\sigma_{k-1}} > w_{\sigma_{k-1}-1}$ and $$w_{\sigma_{k-1}-1} > \cdots > w_{\sigma_{k-1} + 2} > w_{\sigma_{k-1} + 1}$$ 

\[
\ytableausetup{mathmode, boxsize=4em}
\begin{ytableau}
w_{n}\\
w_{n-1}\\
\none[\vdots]\\
w_{\sigma_{k-1}+2} & w_{\sigma_{k-1}}\\
w_{\sigma_{k-1} + 1} & w_{\sigma_{k-1}-1}
\end{ytableau}
\]
This process then continues similarly for inserting the reverse of the second row of the tableau, and gives us the second column as the transpose of the second row of $T$:
\[
\ytableausetup{mathmode, boxsize=4em}
\begin{ytableau}
w_{n}\\
\none[\vdots] & w_{\sigma_{k-1}}\\
\none[\vdots] & \none[\vdots]\\
w_{\sigma_{k-1} +3} & w_{\sigma_{k-2} + 3}\\
w_{\sigma_{k-1} + 2} & w_{\sigma_{k-2} + 2}\\
w_{\sigma_{k-1} + 1} & w_{\sigma_{k-2} + 1}
\end{ytableau}
\]
By a similar argument, each successive row will be inserted as the next column, and so we have shown that the insertion tableau is the transpose $T'$.
\end{proof}

As an example, consider the tableau
$$\young(6,34,1257)$$
Its reading word is $(6,3,4,1,2,5,7)$ so the reverse reading word is $(7,5,2,1,4,3,6)$.  Applying RSK to this reversed reading word, we obtain the pair:
$$\left(\,\young(7,5,24,136)\hspace{0.3cm},\hspace{0.3cm}\young(4,3,26,157)\,\right)$$
Notice that the insertion tableau is the transpose of the original tableau, as indicated by Theorem \ref{theorem:transpose-insertion}.  The recording tableau is the tableau $Q_\lambda$ for this shape (as defined in Lemma \ref{lem:col-reading}), formed by filling the columns upwards from left to right with the numbers $1,2,\ldots,7$ in order.  Indeed, the proof of Theorem \ref{theorem:transpose-insertion} shows that the recording tableau is always equal to $Q_\lambda$, when inserting the reverse reading word of a tableau of shape $\lambda'$, since the columns of the insertion tableau are built up from bottom to top and in order from left to right.  We therefore have the following corollary.

\begin{corollary}\label{cor:increasing-columns}
  For any partition $\lambda$, let $Q_\lambda$ be the tableau of shape $\lambda$ formed by filling the first column with $1,2,\ldots,\lambda_1'$, then the second column with $\lambda_1'+1,\ldots,\lambda_1'+\lambda_2'$, and so on.
  
  Then when RSK is applied to the reversed reading word of any SYT of the transpose shape $\lambda'$, the recording tableau is always $Q_\lambda$.
\end{corollary}

We now have the tools to analyze the dynamics of the iterated map $r:S_n\to S_n$.  Recall the definition of a $2$-cycle from Definition \ref{def:cycle}.

\begin{theorem}\label{thm:two-cycle}
    For any starting permutation $\pi\in S_n$, two or less iterations of the map $r$ will result in a tableau in a cycle with length of at most two.
\end{theorem}

\begin{proof}
  Let $\pi$ be a permutation, let $(S,T)=\mathrm{RSK}(\pi)$, and let $\lambda$ be the shape of the tableaux  $S$ and $T$. Applying the map $r$ gives a permutation $\pi'$ whose RSK pair is  $\mathrm{RSK}(\pi')=(S',Q_{\lambda'})$ where $S'$ has shape $\lambda'$, the transpose of the shape $\lambda$ (by Lemma \ref{lem:transposition}). Here, $Q_{\lambda'}$ has its columns filled in order by $1, 2, 3, \ldots, n$ from bottom to top, starting with the leftmost column, as described by Corollary \ref{cor:increasing-columns}.
  
  A second application of $r$ results in a permutation $\pi''$ satisfying $\mathrm{RSK}(\pi'')=(S'',Q_\lambda)$ where $S''$ and $Q_\lambda$ have the original shape $\lambda$.  Here, $S''$ is the transpose $Q_{\lambda'}'$ of $Q_{\lambda'}$ by Theorem \ref{theorem:transpose-insertion} and thus has its rows filled with the numbers $1, 2, 3, \ldots, n$ in order, from right to left in each row and then from bottom to top.
  
  This shows that any starting permutation of shape $\lambda$ will map to $(Q_{\lambda'}',Q_\lambda)$ within two applications of $r$.  Notice that another application of $r$ then yields the pair $(Q_{\lambda}',Q_{\lambda'})$, and it then repeats in a $2$-cycle (or a $1$-cycle if the two pairs above are equal).  Thus, our ending permutation $\pi''$ will loop back to itself after two iterations and therefore must be part of a cycle with length of at most two.
\end{proof}

To observe how this process works, examine the dynamics of the permutation $\pi = (3,5,1,2,6,7,4)$. Applying RSK results in the SYT pair $(S, T)$, with $S$ on the left and $T$ on the right.
$$\left(\,\young(356,1247)\hspace{0.3cm},\hspace{0.3cm}\young(347,1256)\,\right)$$
We take the reverse reading word of $T$, which is $\pi' = (6,5,2,1,7,4,3)$. Applying RSK to this will generate this new pair of SYT $S', T'$.
$$\left(\,\young(6,57,24,13)\hspace{0.3cm},\hspace{0.3cm}\young(4,37,26,15)\,\right)$$
Here you can notice that $S'$ is the transpose of $T$ and that $T'$ has its columns filled in order from bottom to top and left to right. Taking the reverse reading word of $T'$ gives $\pi'' = (5,1,6,2,7,3,4)$. Applying RSK generates the SYT pair $(S'', T'')$ in the same shape as the original.
$$\left(\, \young(567,1234)\hspace{0.3cm}, \hspace{0.3cm}\young(246,1357) \, \right)$$
Notice that this is precisely the pair $(Q_{(4,3)}',Q_{(4,3)})$.  Applying $r$ one more time yields the pair$$\left(\,\young(7,56,34,12) \hspace{0.3cm}, \hspace{0.3cm}\young(4,37,26,15)\,\right)$$ and the process then repeats in a $2$-cycle.

In some situations, this $2$-cycle collapses to a fixed point.  A partition $\lambda$ is \textbf{self-conjugate} if $\lambda'=\lambda$.

\begin{corollary}
  For the process described above, all permutations whose insertion tableaux have a self-conjugate shape result in a $1$-cycle, and all those with a non self-conjugate shape result in a $2$-cycle.
\end{corollary}

\begin{proof}
Using the same terminology as the previous proof, applying RSK to the permutation $\pi''$ results in the SYT pair $(Q_{\lambda'}', Q_\lambda)$ of shape $\lambda$.  Then taking the reverse reading word of $Q_\lambda$ and applying RSK again yields $(Q_\lambda',Q_{\lambda'})$.

Notice that if the shape $\lambda$ is self-conjugate then $Q_{\lambda}'=Q_{\lambda'}'$ and $Q_{\lambda}=Q_{\lambda'}$, so the two pairs are equal and $\pi''$ is a fixed point (forming a $1$-cycle).

 Additionally, if $\lambda$ is not self-conjugate, it must be in a $2$-cycle because after one step, it is in a different shape and Theorem \ref{thm:two-cycle} states that it is at most a $2$-cycle.
\end{proof}

We illustrate these observations in Figure \ref{fig:second-diagram} for the case of $S_4$.

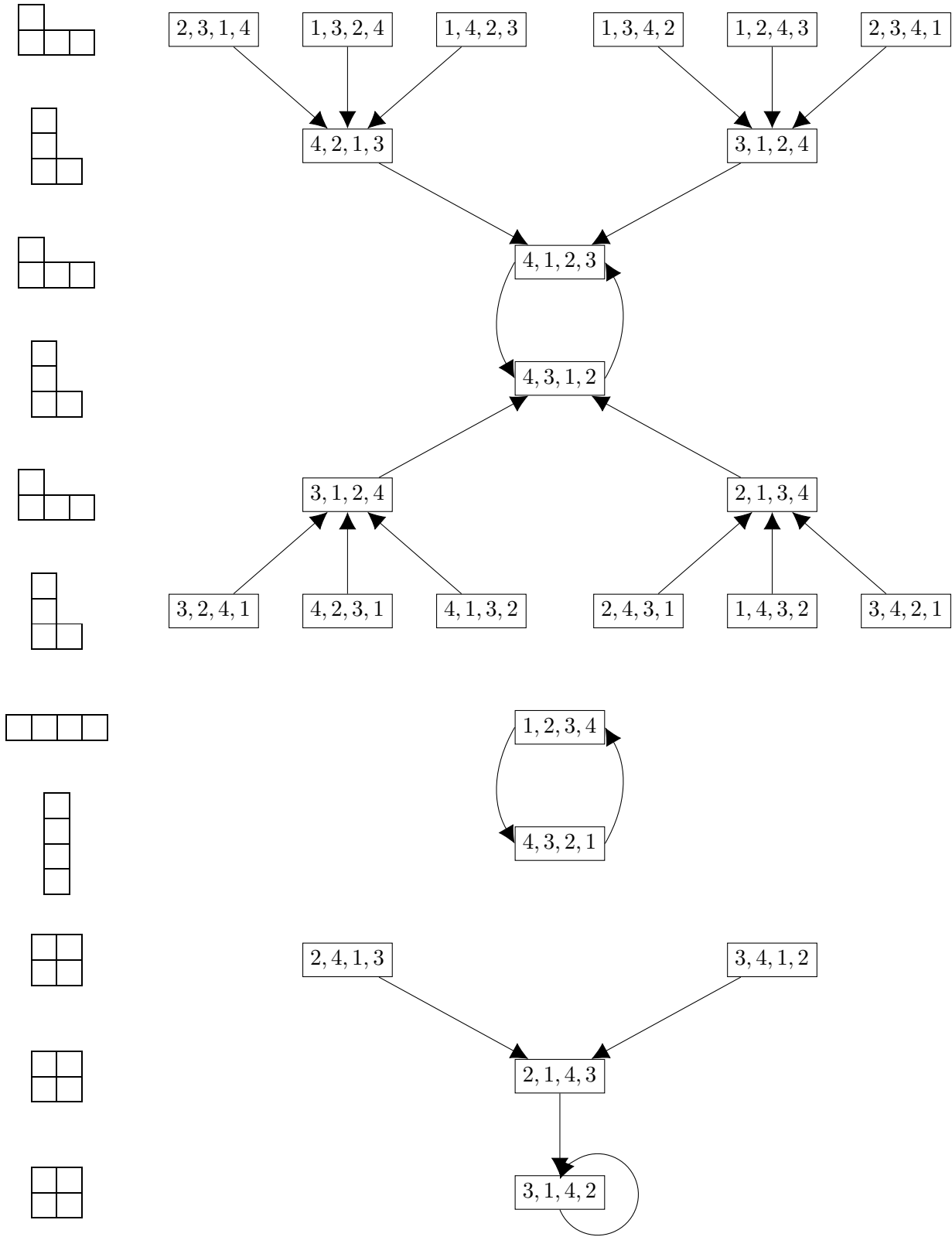
\begin{figure}
    \centering

\tikzstyle{permblock} = [draw, text centered, minimum height=1em]
\tikzstyle{block}  = [rectangle, draw, text width=3.5cm, text centered, minimum height=1em]
\tikzstyle{lblock} = [rectangle, draw, text width=5cm, text centered, minimum height=1em]
\tikzstyle{rblock} = [rectangle, draw, text width=3.5cm, text centered, minimum height=1em]
\tikzstyle{arrow} = [thick,->,>=triangle 45]
\begin{tikzpicture}[node distance=2cm]
\node (shape131) {$\yng(1,3)$};
\node (rw11) [permblock, right of=shape131, xshift=0.7cm] {$2,3,1,4$};
\node (rw12) [permblock, right of=rw11, xshift=0.3cm] {$1,3,2,4$};
\node (rw13) [permblock, right of=rw12, xshift=0.3cm] {$1,4,2,3$};
\node (rw14) [permblock, right of=rw13, xshift=0.7cm] {$1,3,4,2$};
\node (rw15) [permblock, right of=rw14, xshift=0.3cm] {$1,2,4,3$};
\node (rw16) [permblock, right of=rw15, xshift=0.3cm] {$2,3,4,1$};
\node (shape311)[below of=shape131]{$\yng(1,1,2)$};
\node (rw21) [permblock, right of=shape311, xshift=3cm] {$4,2,1,3$};
\node (rw22) [permblock, right of=rw21, xshift=5.3cm] {$3,1,2,4$};
\node (shape132)[below of=shape311]{$\yng(1,3)$};
\node (rw31) [permblock, right of=shape132, xshift=6.65cm] {$4,1,2,3$};
\node (shape312)[below of=shape132]{$\yng(1,1,2)$};
\node (rw41) [permblock, right of=shape312, xshift=6.65cm] {$4,3,1,2$};
\node (shape133)[below of=shape312]{$\yng(1,3)$};
\node (rw51) [permblock, right of=shape133, xshift=3cm] {$3,1,2,4$};
\node (rw52) [permblock, right of=rw51, xshift=5.3cm] {$2,1,3,4$};
\node (shape313) [below of=shape133] {$\yng(1,1,2)$};
\node (rw61) [permblock, right of=shape313, xshift=0.7cm] {$3,2,4,1$};
\node (rw62) [permblock, right of=rw61, xshift=0.3cm] {$4,2,3,1$};
\node (rw63) [permblock, right of=rw62, xshift=0.3cm] {$4,1,3,2$};
\node (rw64) [permblock, right of=rw63, xshift=0.7cm] {$2,4,3,1$};
\node (rw65) [permblock, right of=rw64, xshift=0.3cm] {$1,4,3,2$};
\node (rw66) [permblock, right of=rw65, xshift=0.3cm] {$3,4,2,1$};
\draw[->] (rw11) -- (rw21) ;
\draw[->] (rw12) -- (rw21) ;
\draw[->] (rw13) -- (rw21) ;
\draw[->] (rw14) -- (rw22) ;
\draw[->] (rw15) -- (rw22) ;
\draw[->] (rw16) -- (rw22) ;
\draw[->] (rw21) -- (rw31) ;
\draw[->] (rw22) -- (rw31) ;
\path[every node/.style={font=\sffamily\small}]
    (rw31.west) edge[bend right, ->] node [left] {} (rw41.west);
\path[every node/.style={font=\sffamily\small}]
    (rw41.east) edge[bend right, ->] node [left] {} (rw31.east);

\draw[->] (rw51) -- (rw41) ;
\draw[->] (rw52) -- (rw41) ;
\draw[->] (rw61) -- (rw51) ;
\draw[->] (rw62) -- (rw51) ;
\draw[->] (rw63) -- (rw51) ;
\draw[->] (rw64) -- (rw52) ;
\draw[->] (rw65) -- (rw52) ;
\draw[->] (rw66) -- (rw52) ;

\node (shape041)[below of=shape313]{$\yng(4)$};
\node (rw71) [permblock, right of=shape041, xshift=6.65cm] {$1,2,3,4$};
\node (shape401)[below of=shape041]{$\yng(1,1,1,1)$};
\node (rw81) [permblock, right of=shape401, xshift=6.65cm] {$4,3,2,1$};
\path[every node/.style={font=\sffamily\small}]
    (rw71.west) edge[bend right, ->] node [left] {} (rw81.west);
\path[every node/.style={font=\sffamily\small}]
    (rw81.east) edge[bend right, ->] node [left] {} (rw71.east);

\node (shape221)[below of=shape401]{$\yng(2,2)$};
\node (rw91) [permblock, right of=shape221, xshift=3cm] {$2,4,1,3$};
\node (rw92) [permblock, right of=rw91, xshift=5.3cm] {$3,4,1,2$};
\node (shape222)[below of=shape221]{$\yng(2,2)$};
\node (rw101) [permblock, right of=shape222, xshift=6.65cm] {$2,1,4,3$};
\node (shape223)[below of=shape222]{$\yng(2,2)$};
\node (rw111) [permblock, right of=shape223, xshift=6.65cm] {$3,1,4,2$};
\draw[->] (rw91) -- (rw101) ;
\draw[->] (rw92) -- (rw101) ;
\draw[->] (rw101) -- (rw111) ;
\draw [->] (rw111.south)arc(-158:158:7mm);
\end{tikzpicture}
    \caption{The dynamical system $r$ formed by iterating the RSK algorithm on the reversed reading word of the recording tableau.  The shape of the insertion tableau of each permutation is shown at left.}
    \label{fig:second-diagram}
\end{figure}

Finally, we enumerate the fixed points and $2$-cycles for a given size $n$.  By the above corollary, we have that there is exactly one fixed point for each self-conjugate partition of $n$, and one $2$-cycle for every pair of transposed partitions that are not self-conjugate.  

It is known that the number of self-conjugate shapes of size $n$ is equal to the number of partitions of $n$ into distinct odd parts (see \cite[Proposition 1.8.4]{Stanley1}).  Let $p_{do}(n)$ be the number of partitions of $n$ into distinct odd parts; then it follows that there are $p_{do}(n)$ fixed points and $\frac{1}{2}\left(p(n)-p_{do}(n)\right)$ cycles of length $2$, where $p(n)$ is the total number of partitions of $n$.

\section{Acknowledgments}

We thank Chris Peterson for giving the faculty author the idea for this project.

We would also like to thank the 2020 Online Prove it!\ Math Academy instructors Richard Fried, Bryan Gillespie, Aileen Ma, Ken G.\ Monks, and Ken M.\ Monks  for their helpful feedback and input throughout this project.  

Finally, we thank the referees for their significant feedback.


\begin{thebibliography}{99}

    \bibitem{BumpSchilling} D.~Bump and A.~Schilling, \textit{Crystal Bases: Representations and Combinatorics}, World Scientific (2017).

    \bibitem{Haiman} M.~Haiman, Dual Equivalence with Applications, Including a Conjecture of Proctor, \textit{Discrete Math.}, Vol. 9, Issues 1--3 (1992), pp.\ 79--113.

    \bibitem{Knuth} D.\ Knuth, Permutations, matrices, and generalized Young tableaux, \textit{Pacific J.\ Math.}, Vol.\ \textit{34}, Number 3 (1970), pp.\ 709--727.

    \bibitem{Fulton} W.~Fulton, \textit{Young tableaux}, London Math.\ Soc.\ Student Texts \textbf{35}, Cambridge University Press (1997).

    \bibitem{Robinson}  G.\ de B.\ Robinson, On the Representations of the Symmetric Group, \textit{Amer.\ J.\ Math.}, Vol.\ \textbf{60}, Number 3 (1938), pp.\  745--760.

    \bibitem{Sagan} B.~Sagan, The Symmetric Group, 2nd ed., Springer, New York, 2001.

    \bibitem{Schensted} C.~Schensted, Longest increasing and decreasing subsequences, \textit{Canad.\ J.\ Math.} \textbf{13}  (1961), pp.\  179--191. 
    \bibitem{Stanley1} R.~Stanley, \textit{Enumerative Combinatorics, Vol. 1}, Cambridge Studies in Advanced Mathematics \textbf{49}, Cambridge University Press, 1997.
    \bibitem{Stanley} R.~Stanley, \textit{Enumerative Combinatorics, Vol. 2}, Cambridge Studies in Advanced Mathematics \textbf{62}, Cambridge University Press, 1999.
\end{thebibliography}
\end{document}